\documentclass[12pt]{amsart}

\usepackage[T1]{fontenc}
\usepackage[utf8]{inputenc}
\usepackage{amsmath,amssymb,amsthm,mathtools}
\usepackage{hyperref}

\newtheorem{theorem}{Theorem}[section]

\newtheorem{lemma}[theorem]{Lemma}
\newtheorem{corollary}[theorem]{Corollary}
\theoremstyle{definition}
\newtheorem{definition}[theorem]{Definition}
\theoremstyle{remark}
\newtheorem{remark}[theorem]{Remark}

\DeclareMathOperator{\Gal}{Gal}
\DeclareMathOperator{\Br}{Br}
\DeclareMathOperator{\inv}{inv}

\newcommand{\Fp}{\mathbb F_p}

\newcommand{\Hp}{H_{p^3}}
\newcommand{\E}{\mathcal E_p}

\title[Common-slot chains and Heisenberg central products]{From Common-Slot Chains to Heisenberg Central Products over Global Fields}
\author[M. Palaisti]{M. Palaisti}
\date{}

\begin{document}

\begin{abstract}
Let $F$ be a global field and let $p$ be an odd prime with $p\neq\operatorname{char}F$. Assuming first that $\mu_p\subset F$, we record in a uniform global-field form the length-four common-slot chain for equal degree-$p$ symbol classes and emphasize the signed normalization adapted to explicit norm constructions. Thus, from
\[(a,b)_p=(c,d)_p\in\Br(F)[p]\]
one obtains $x,y\in F^\times$ such that
\[(a,b)_p=(x^{-1},b)_p=(x,y)_p=(c^{-1},y)_p=(c,d)_p.\]
For number fields, the underlying chain is the length-four chain lemma of Gille--Szamuely, based on Tate's simultaneous local--global theorem. The point developed here is that its signed form yields four compatible norm equations that can be used constructively. For the extraspecial central product $H_{p^3}*H_{p^3}$ over a $C_p^4$-Kummer extension, the central-embedding obstruction is $(a,b)_p-(c,d)_p$, while the four norm equations supplied by the chain assemble, under a natural independence hypothesis on the auxiliary Kummer classes, into an explicit factorized radical realization of the central product. Finally, when the ground field does not contain $\mu_p$, we show by a restriction--corestriction argument that the central-embedding obstruction is detected after passage to the cyclotomic extension $F(\mu_p)$. Over that field the problem is Kummer, and the obstruction is again the difference of the two symbol classes.
\end{abstract}

\maketitle



\section{Introduction}

Common-slot and chain phenomena for symbol algebras have a long history. For
quaternion algebras, the common-slot principle produces particularly short
chains, whereas in higher degree the structure is more delicate. Rost proved
a chain lemma for Kummer elements in degree $3$, with a corresponding
common-slot consequence for cyclic algebras of degree $3$; see \cite{Rost}.
More recent work has continued to investigate the scope and limitations of
common-slot properties for symbol algebras; see, for example,
\cite{Sivatski}.

For arithmetic fields, Tate's work on $K_2$ and Galois cohomology provides a
powerful simultaneous local--global principle. In the number-field case,
Gille and Szamuely use this input to prove a length-four chain lemma
\cite[Lemma~7.6.11]{GilleSzamuely}. In symbol-algebra notation, if $F$
contains $\mu_p$ and
\[(a,b)_p=(c,d)_p\in\Br(F)[p],\]
then there exist $x,y\in F^\times$ such that
\begin{equation}\label{eq:signed-chain-intro}
	(a,b)_p=(x^{-1},b)_p=(x,y)_p=(c^{-1},y)_p=(c,d)_p.
\end{equation}
Indeed, the chain displayed in \cite[Lemma~7.6.11]{GilleSzamuely} is
\[(b,x)_p=(x,y)_p=(y,c)_p=(c,d)_p,\]
which is equivalent to \eqref{eq:signed-chain-intro} by skew-symmetry. The
same local--global argument applies to the global fields considered here:
the proof uses the nondegenerate local Hilbert-symbol pairing, weak
approximation, and Tate's global-field theorem. We include the argument in
Section~2 in order to keep the paper self-contained at the point where the
chain is used and, more importantly, to retain precise control of the signed
normalization needed for the subsequent construction.

The central observation of this paper is that, in the embedding problem
considered below, the common-slot chain carries substantially more
constructive information than the vanishing of the Brauer obstruction alone.
In the signed form \eqref{eq:signed-chain-intro}, the equality of the two
symbol classes yields four compatible norm relations:
\begin{align*}
	ax &\in N_{F(\sqrt[p]{b})/F}\bigl(F(\sqrt[p]{b})^\times\bigr),\\
	x^{-1} &\in N_{F(\sqrt[p]{by})/F}\bigl(F(\sqrt[p]{by})^\times\bigr),\\
	y^{-1} &\in N_{F(\sqrt[p]{xc})/F}\bigl(F(\sqrt[p]{xc})^\times\bigr),\\
	yd &\in N_{F(\sqrt[p]{c})/F}\bigl(F(\sqrt[p]{c})^\times\bigr).
\end{align*}
These four equations provide the arithmetic data from which we construct an
explicit radical realization of a Heisenberg central product.

More precisely, the group-theoretic object of interest is
\[\E=\Hp*\Hp,\]
the central product obtained by identifying the centers of two Heisenberg
groups of order $p^3$ with opposite orientation. It is an extraspecial group
of order $p^5$ and exponent $p$, fitting into a central extension
\[1\longrightarrow C_p\longrightarrow\E\longrightarrow C_p^4
\longrightarrow1.\]
We consider the corresponding embedding problem over a Kummer extension
\[K=F\bigl(\sqrt[p]{a},\sqrt[p]{b},\sqrt[p]{c},\sqrt[p]{d}\bigr).\]
For the presentation of $\E$ used here, the standard central-embedding
obstruction specializes to
\[(a,b)_p-(c,d)_p.\]
Hence the condition
\[(a,b)_p=(c,d)_p\]
is exactly the solvability condition for the embedding problem.

The obstruction criterion sits naturally within the classical theory of
central $p$-extensions. Massy developed explicit obstruction and construction
methods for central $p$-extensions of elementary abelian $p$-extensions
\cite{Massy}; Swallow subsequently gave general explicit constructions for
central embedding problems and for central $p$-extensions of
$(p,p,\ldots,p)$-type Galois groups
\cite{SwallowConstructible,SwallowCentral}. Ledet's monograph
\cite{Ledet} gives a systematic treatment of Brauer-type embedding problems
and the obstruction formalism used below.

The contribution of the present paper is a chain-adapted explicit
construction in which the arithmetic and group-theoretic data are linked
directly. Under a natural independence
hypothesis on the auxiliary Kummer classes $x$ and $y$, the four norm
equations supplied by \eqref{eq:signed-chain-intro} can be assembled into a
single factorized radical extension. The lifted Kummer generators in this
extension satisfy precisely the commutator relations defining $\Hp*\Hp$.
Thus the main structural passage of the paper is
\begin{eqnarray*}
	\text{equality of symbols}
	&\Longrightarrow&
	\text{common-slot chain}
	\Longrightarrow\\
	\text{four compatible norm equations}
	&\Longrightarrow&
	\text{factorized radical realization of }\Hp*\Hp.
\end{eqnarray*}
In this form, the obstruction theory determines when the desired extension
exists, while the common-slot chain supplies, in the generic auxiliary case,
a particularly transparent way to construct it.

This construction places the general obstruction theory in a form adapted
to the present Heisenberg central product. Its advantage is that the Brauer
obstruction, the common-slot chain, the associated norm equations, and the
resulting commutator relations are exhibited in a single compatible
calculation. The same viewpoint also extends naturally
from number fields to the global fields considered here. When the ground
field does not contain $\mu_p$, we pass to the cyclotomic extension, formulate
the obstruction in Kummer form there, and then descend solvability back to
the original field.

The paper is organized as follows. Section~2 records the local and global
common-slot chain and derives the signed norm relations used later.
Section~3 specializes the standard central-embedding obstruction to
$\Hp*\Hp$ and gives the chain-adapted radical construction in the generic
auxiliary case. Section~4 explains the corresponding cyclotomic descent when
the ground field does not contain $\mu_p$.

\section{Common-slot chains over global fields}

Throughout this section, $F$ is a global field, $p$ is an odd prime with
$p\neq\operatorname{char}F$, and $\mu_p\subset F$. We write $(u,v)_p$ for
the degree-$p$ symbol algebra and use additive notation for $\Br(F)[p]$ when
signs are involved. Thus
\[
 (u,v)_p=-(v,u)_p.
\]

\subsection{Tate's simultaneous local--global lemma}

The global input in the chain argument is a simultaneous local--global
theorem of Tate. We state only the form needed here. In the number-field
setting, this is the same input used in the proof of the length-four chain
lemma of Gille--Szamuely.

\begin{theorem}[Tate]\label{thm:Tate}
Let $\alpha_1,\ldots,\alpha_r\in\Br(F)[p]$ and let
$a_1,\ldots,a_r\in F^\times$. Suppose that for every place $v$ of $F$
there is $x_v\in F_v^\times$ such that
\[
 (a_i,x_v)_{p,v}=(\alpha_i)_v,
 \qquad i=1,\ldots,r.
\]
Then there exists $x\in F^\times$ satisfying
\[
 (a_i,x)_p=\alpha_i,
 \qquad i=1,\ldots,r.
\]
\end{theorem}

This is \cite[Lemma~5.2]{Tate}, translated into the present
symbol-algebra notation. We shall use it only for $r=2$.

\subsection{The local chain}

Let $k$ be a nonarchimedean local field containing $\mu_p$. The local
invariant identifies $\Br(k)[p]$ with
$\frac1p\mathbb Z/\mathbb Z\cong\Fp$, and the Hilbert symbol induces a
nondegenerate alternating pairing
\[
 B_k:k^\times/k^{\times p}\times k^\times/k^{\times p}
 \longrightarrow\Fp,
 \qquad
 B_k(\bar u,\bar v)=\inv_k((u,v)_p).
\]
At archimedean places the $p$-primary Brauer group is trivial for odd $p$,
so no separate argument is needed.

\begin{lemma}\label{lem:local-chain}
Let $k$ be a local field with $p\neq\operatorname{char}k$ and
$\mu_p\subset k$. Suppose
\[
 (a,b)_p=(c,d)_p=\alpha\in\Br(k)[p].
\]
Then there exist $X,Y\in k^\times$ such that
\[
 (a,b)_p=(X,b)_p=(X,Y)_p=(c,Y)_p=(c,d)_p.
\]
\end{lemma}

\begin{proof}
If $\alpha=0$, take $X=Y=1$. Assume $\alpha\neq0$ and write
$\lambda=\inv_k(\alpha)\in\Fp^\times$. Put
\[
 V=k^\times/k^{\times p},
\]
and let $B=B_k$. Since
\[
 B(\bar a,\bar b)=\lambda\neq0,
\]
the classes $\bar a$ and $\bar b$ are linearly independent; in particular
$\dim_{\Fp}V\ge2$.

Consider the affine hyperplane
\[
 H=\{u\in V:B(u,\bar b)=\lambda\}.
\]
We claim that one may choose $\bar X\in H$ outside the line
$\Fp\bar c$. The hyperplane $H$ does not contain $0$, whereas
$\Fp\bar c$ is a linear subspace through $0$. If $\dim_{\Fp}V=2$, both
$H$ and a one-dimensional subspace have $p$ elements, so containment would
force equality, which is impossible because only the latter contains $0$.
If $\dim_{\Fp}V>2$, then $|H|=p^{\dim V-1}>p=|\Fp\bar c|$. Hence such an
$\bar X$ exists.

Because $\bar X$ and $\bar c$ are linearly independent, nondegeneracy of
$B$ implies that the two linear functionals
\[
 B(\bar X,-),\qquad B(\bar c,-)
\]
are independent. Hence there exists $\bar Y\in V$ such that
\[
 B(\bar X,\bar Y)=\lambda,
 \qquad
 B(\bar c,\bar Y)=\lambda.
\]
Choosing representatives $X,Y\in k^\times$ gives the desired chain.
\end{proof}

\subsection{Globalization}

For number fields, the following statement is the length-four chain lemma of
Gille--Szamuely \cite[Lemma~7.6.11]{GilleSzamuely}. We give the argument in
the present notation because the signed normalization and the intermediate
local data are used directly in Section~3. The same proof applies to global
function fields under our standing hypothesis $p\neq\operatorname{char}F$.

\begin{theorem}\label{thm:global-chain}
Let $F$ be a global field, let $p$ be an odd prime with
$p\neq\operatorname{char}F$, and assume $\mu_p\subset F$. If
\[
 (a,b)_p=(c,d)_p\in\Br(F)[p],
\]
then there exist $X,Y\in F^\times$ such that
\begin{equation}\label{eq:ordinary-chain}
 (a,b)_p=(X,b)_p=(X,Y)_p=(c,Y)_p=(c,d)_p.
\end{equation}
Equivalently, there exist $x,y\in F^\times$ such that
\begin{equation}\label{eq:signed-chain}
 (a,b)_p=(x^{-1},b)_p=(x,y)_p=(c^{-1},y)_p=(c,d)_p.
\end{equation}
\end{theorem}

\begin{proof}
Set
\[
 \alpha=(a,b)_p=(c,d)_p,
\]
and, for each place $v$ of $F$, let $\alpha_v$ denote the image of $\alpha$
in $\Br(F_v)[p]$. Put
\[
 S=\{v:\alpha_v\neq0\}.
\]
The set $S$ is finite.

For every $v\in S$, Lemma~\ref{lem:local-chain} gives
$X_v,Y_v\in F_v^\times$ such that
\[
 (X_v,b)_{p,v}=(X_v,Y_v)_{p,v}=(c,Y_v)_{p,v}=\alpha_v.
\]
Since $(c,Y_v)_{p,v}=(c,d)_{p,v}$, we have
\[
 (c,Y_v/d)_{p,v}=0.
\]
By the norm criterion for cyclic algebras, there exists
\[
 t_v\in F_v(\sqrt[p]{c})^\times
\]
with
\[
 \frac{Y_v}{d}=N_{F_v(\sqrt[p]{c})/F_v}(t_v).
\]

For $v\in S$ the class $\alpha_v$ is nonzero. Since
$(c,d)_{p,v}=\alpha_v$, the element $c$ cannot lie in $F_v^{\times p}$;
otherwise every symbol with first slot $c$ would vanish. Hence
$F_v(\sqrt[p]{c})/F_v$ is a field extension of degree $p$. In particular,
there is a unique place of $F(\sqrt[p]{c})$ above $v$ giving this completion.

The subgroup of $p$-th powers in a nonarchimedean local field is open. By
weak approximation in $F(\sqrt[p]{c})$, choose
\[
 t\in F(\sqrt[p]{c})^\times
\]
so close to the finitely many $t_v$ at the places above $S$ that, for every
$v\in S$, the quotient $t/t_v$ is a local $p$-th power. Compatibility of
the global norm with completion then gives
\[
 \frac{N_{F(\sqrt[p]{c})/F}(t)}
 {N_{F_v(\sqrt[p]{c})/F_v}(t_v)}
 \in F_v^{\times p}.
\]
Put
\[
 Y=d\,N_{F(\sqrt[p]{c})/F}(t).
\]
Since a norm from $F(\sqrt[p]{c})$ gives a split cyclic algebra,
\[
 (c,Y)_p=(c,d)_p=\alpha.
\]
Moreover, for every $v\in S$ the quotient $Y/Y_v$ is a $p$-th power in
$F_v^\times$, and consequently
\[
 (X_v,Y)_{p,v}=\alpha_v.
\]

For $v\notin S$ we have $\alpha_v=0$, and $X_v=1$ is a simultaneous local
solution of
\[
 (X_v,b)_{p,v}=\alpha_v,
 \qquad
 (X_v,Y)_{p,v}=\alpha_v.
\]
Thus, for every place $v$, there is $X_v\in F_v^\times$ satisfying
\[
 (b,X_v)_{p,v}=-\alpha_v,
 \qquad
 (Y,X_v)_{p,v}=-\alpha_v.
\]
Applying Theorem~\ref{thm:Tate} with
$a_1=b$, $a_2=Y$, and $\alpha_1=\alpha_2=-\alpha$ gives
$X\in F^\times$ such that
\[
 (b,X)_p=-\alpha,
 \qquad
 (Y,X)_p=-\alpha.
\]
By skew-symmetry,
\[
 (X,b)_p=\alpha,
 \qquad
 (X,Y)_p=\alpha.
\]
Together with $(c,Y)_p=\alpha$, this proves
\eqref{eq:ordinary-chain}.

Finally, set
\[
 x=X^{-1},\qquad y=Y^{-1}.
\]
Since inversion in both slots leaves a symbol unchanged,
\[
 (x,y)_p=(X,Y)_p,
 \qquad
 (c^{-1},y)_p=(c,Y)_p,
\]
and \eqref{eq:signed-chain} follows.
\end{proof}

\begin{remark}[Norm data from the signed chain]\label{rem:norm-data}
The signed chain \eqref{eq:signed-chain} gives the following four norm
relations:
\begin{align}
 ax &\in N_{F(\sqrt[p]{b})/F}
    \bigl(F(\sqrt[p]{b})^\times\bigr),\label{eq:norm1}\\
 x^{-1} &\in N_{F(\sqrt[p]{by})/F}
    \bigl(F(\sqrt[p]{by})^\times\bigr),\label{eq:norm2}\\
 y^{-1} &\in N_{F(\sqrt[p]{xc})/F}
    \bigl(F(\sqrt[p]{xc})^\times\bigr),\label{eq:norm3}\\
 yd &\in N_{F(\sqrt[p]{c})/F}
    \bigl(F(\sqrt[p]{c})^\times\bigr).\label{eq:norm4}
\end{align}
For example, the equality $(x^{-1},b)_p=(x,y)_p$ gives
\[
 (x,by)_p=0,
\]
so, by skew-symmetry, $(by,x^{-1})_p=0$, which is exactly
\eqref{eq:norm2}. Likewise,
\[
 (a,b)_p=(x^{-1},b)_p
 \quad\Longrightarrow\quad
 (ax,b)_p=0,
\]
\[
 (x,y)_p=(c^{-1},y)_p
 \quad\Longrightarrow\quad
 (xc,y)_p=0,
\]
and
\[
 (c^{-1},y)_p=(c,d)_p
 \quad\Longrightarrow\quad
 (c,yd)_p=0,
\]
which yield \eqref{eq:norm1}, \eqref{eq:norm3}, and \eqref{eq:norm4} by
the norm criterion.
\end{remark}

\section{The Heisenberg central product: obstruction and construction}

We first identify the central-embedding obstruction for the Heisenberg
central product and obtain the exact solvability criterion. We then use the
common-slot chain from Section~2 to convert obstruction vanishing into the
structured norm data underlying the explicit construction of
Section~\ref{subsec:explicit}.

\subsection{The target group and its obstruction}

Let $\Hp$ be the Heisenberg group of order $p^3$ and exponent $p$. We use
the commutator convention
\[
 [g,h]=g^{-1}h^{-1}gh
\]
and the presentation
\[
 \Hp=\langle u,v,\tau\mid
 u^p=v^p=\tau^p=1,\ \tau\text{ central},\ [u,v]=\tau\rangle.
\]

\begin{definition}\label{def:E}
Let
\[
 \E=\Hp*\Hp
\]
be the central product obtained by identifying the centers of two copies of
$\Hp$ with opposite orientation. Equivalently,
\[
 \E=\langle s_a,s_b,s_c,s_d,\tau\mid
 s_a^p=s_b^p=s_c^p=s_d^p=\tau^p=1,
\]
\[
 \tau\text{ central},\quad
 [s_a,s_b]=\tau,\quad [s_c,s_d]=\tau^{-1},
\]
\[
 [s_a,s_c]=[s_a,s_d]=[s_b,s_c]=[s_b,s_d]=1\rangle.
\]
Then $|\E|=p^5$, and there is a central extension
\begin{equation}\label{eq:E-extension}
 1\longrightarrow C_p\longrightarrow\E\longrightarrow C_p^4\longrightarrow1.
\end{equation}
\end{definition}

We recall the standard obstruction formula for central $\mu_p$-embedding
problems over elementary abelian Kummer extensions. This belongs to the
classical explicit theory of central $p$-extensions developed by Massy; we
use the formulation given in \cite[Corollary~6.1.6]{Ledet}; see also
\cite{Massy}.

\begin{theorem}[Central embedding obstruction]\label{thm:obstruction}
Assume $\mu_p\subset F$ and let
\[
 K=F(\sqrt[p]{a_1},\ldots,\sqrt[p]{a_n})
\]
be a $C_p^n$-extension. Let
\[
 1\longrightarrow\mu_p\longrightarrow E\longrightarrow C_p^n\longrightarrow1
\]
be a central extension, and choose lifts $s_i\in E$ of the standard Kummer
generators. Write
\[
 s_j s_i=\zeta^{d_{ij}}s_i s_j\quad(i<j),
 \qquad
 s_i^p=\zeta^{d_i}.
\]
Then the obstruction to the embedding problem is
\[
 \prod_{i=1}^n(a_i,\zeta)_p^{d_i}
 \prod_{i<j}(a_j,a_i)_p^{d_{ij}}
 \in\Br(F)[p].
\]
The embedding problem is solvable if and only if this class is zero.
\end{theorem}

\begin{corollary}\label{cor:obstruction-E}
Let
\[
 K=F(\sqrt[p]{a},\sqrt[p]{b},\sqrt[p]{c},\sqrt[p]{d})
\]
be a $C_p^4$-extension. For the central extension
\eqref{eq:E-extension}, the obstruction is
\[
 (a,b)_p-(c,d)_p.
\]
Consequently, the embedding problem is solvable if and only if
\[
 (a,b)_p=(c,d)_p.
\]
\end{corollary}

\begin{proof}
From $[s_a,s_b]=\tau$ and our commutator convention,
\[
 s_b s_a=\tau^{-1}s_a s_b,
\]
so the $(a,b)$-pair contributes
$(b,a)_p^{-1}=(a,b)_p$. Likewise,
$[s_c,s_d]=\tau^{-1}$ gives
\[
 s_d s_c=\tau s_c s_d,
\]
so the $(c,d)$-pair contributes
$(d,c)_p=-(c,d)_p$. All $p$-power and cross-commutator contributions are
trivial.
\end{proof}

\begin{theorem}\label{thm:global-realizability}
Let $F$ be a global field, let $p$ be an odd prime with
$p\neq\operatorname{char}F$, and assume $\mu_p\subset F$. Let
\[
 K=F(\sqrt[p]{a},\sqrt[p]{b},\sqrt[p]{c},\sqrt[p]{d})
\]
be a $C_p^4$-extension. There exists a Galois extension $L/F$ containing
$K$ such that
\[
 \Gal(L/F)\cong\Hp*\Hp
\]
and the natural map $\Gal(L/F)\to\Gal(K/F)$ is the quotient map in
\eqref{eq:E-extension} if and only if
\[
 (a,b)_p=(c,d)_p.
\]
\end{theorem}

\begin{proof}
This is exactly Corollary~\ref{cor:obstruction-E} together with the standard
interpretation of solvability of the embedding problem.
\end{proof}

\begin{remark}
Theorem~\ref{thm:global-realizability} gives the existence criterion, while
the common-slot chain supplies additional constructive structure. In the
next subsection, its signed normalization produces a coordinated system of
norm equations from which the desired extension is written explicitly.
\end{remark}

\subsection{A chain-adapted explicit construction in the generic case}
\label{subsec:explicit}

We now exploit the four norm equations supplied by the common-slot chain to
construct the Heisenberg central product in a factorized radical form. This
realization is adapted directly to the arithmetic structure developed in
Section~2 and may be viewed alongside the general explicit construction
theories of Massy and Swallow; see
\cite{Massy,SwallowConstructible,SwallowCentral}.

Assume
\[
 (a,b)_p=(c,d)_p,
\]
and let $x,y\in F^\times$ satisfy the signed chain
\eqref{eq:signed-chain}. By Remark~\ref{rem:norm-data}, choose
\begin{align*}
 f_1&\in F(\sqrt[p]{b})^\times,
 &N_b(f_1)&=ax,\\
 f_2&\in F(\sqrt[p]{by})^\times,
 &N_{by}(f_2)&=x^{-1},\\
 f_3&\in F(\sqrt[p]{xc})^\times,
 &N_{xc}(f_3)&=y^{-1},\\
 f_4&\in F(\sqrt[p]{c})^\times,
 &N_c(f_4)&=yd.
\end{align*}
Here $N_u$ denotes the norm from $F(\sqrt[p]{u})$ to $F$.

For a cyclic Kummer extension $F(\sqrt[p]{u})/F$, let $\rho_u$ be the
generator satisfying
\[
 \rho_u(\sqrt[p]{u})=\zeta\sqrt[p]{u}.
\]
Define
\begin{equation}\label{eq:W-def}
 W_u(f)=\prod_{i=0}^{p-2}\rho_u^i(f)^{p-1-i}
 =f^{p-1}\rho_u(f)^{p-2}\cdots\rho_u^{p-2}(f).
\end{equation}

\begin{lemma}\label{lem:W-ratio}
For $f\in F(\sqrt[p]{u})^\times$,
\[
 \frac{\rho_u(W_u(f))}{W_u(f)}=\frac{N_u(f)}{f^p}.
\]
\end{lemma}

\begin{proof}
Shifting the factors in \eqref{eq:W-def} gives
\[
 \rho_u(W_u(f))
 =\rho_u(f)^{p-1}\rho_u^2(f)^{p-2}\cdots\rho_u^{p-1}(f).
\]
Dividing by $W_u(f)$ leaves one copy of each conjugate $\rho_u^i(f)$ for
$1\le i\le p-1$ and $f^{-(p-1)}$, hence
\[
 \frac{\rho_u(W_u(f))}{W_u(f)}
 =\frac{\rho_u(f)\cdots\rho_u^{p-1}(f)}{f^{p-1}}
 =\frac{N_u(f)}{f^p}.
\]
\end{proof}

Set
\[
 w_1=W_b(f_1),\qquad
 w_2=W_{by}(f_2),\qquad
 w_3=W_{xc}(f_3),\qquad
 w_4=W_c(f_4),
\]
and
\[
 w=w_1w_2w_3w_4.
\]

\begin{theorem}\label{thm:explicit}
Assume, in addition to the hypotheses of
Theorem~\ref{thm:global-realizability}, that
\[
 (a,b)_p=(c,d)_p
\]
and that the six classes
\[
 a,b,c,d,x,y\in F^\times/F^{\times p}
\]
are linearly independent. Put
\[
 M=F(\sqrt[p]{a},\sqrt[p]{b},\sqrt[p]{c},\sqrt[p]{d},
   \sqrt[p]{x},\sqrt[p]{y}).
\]
Choose $r\in F^\times$ such that $rw\notin M^{\times p}$, put $w'=rw$,
and set
\[
 \widetilde L=M(\sqrt[p]{w'}).
\]
Then
\[
 \Gal(\widetilde L/F)\cong(\Hp*\Hp)\times C_p^2.
\]
If $A\cong C_p^2$ denotes the factor generated by the two auxiliary Kummer
directions $x$ and $y$, then
\[
 L=\widetilde L^A
\]
contains $K$ and satisfies
\[
 \Gal(L/F)\cong\Hp*\Hp.
\]
\end{theorem}

\begin{proof}
First we justify the choice of $r$. By elementary Kummer theory, the kernel
of the natural map
\[
 F^\times/F^{\times p}\longrightarrow M^\times/M^{\times p}
\]
is precisely the six-dimensional subspace generated by the classes of
$a,b,c,d,x,y$. Indeed, if $u\in F^\times$ becomes a $p$-th power in $M$,
then $F(\sqrt[p]{u})$ is a degree-$p$ subextension of $M/F$, and hence its
Kummer class lies in the span of the six defining classes; the converse is
immediate.

Moreover, $F^\times/F^{\times p}$ is infinite for a global field. For
example, given any finite set of classes, choose a finite place outside the
supports of representatives of those classes and use weak approximation to
produce an element having valuation $1$ at that place. Its class cannot lie
in the prescribed finite-dimensional subspace. The set of classes
$[r]\in F^\times/F^{\times p}$ for which $rw\in M^{\times p}$ is either
empty or a single coset of the above kernel: if both $r_1w$ and $r_2w$ are
$p$-th powers in $M$, then $r_1/r_2$ is a $p$-th power in $M$. Hence we may
choose $r\in F^\times$ with $rw\notin M^{\times p}$. Multiplying $w$ by
$r\in F^\times$ does not change any quotient $\sigma(w)/w$.

Let
\[
 z=\sqrt[p]{w'}.
\]
Since $w'\notin M^{\times p}$, the extension $\widetilde L/M$ has degree
$p$. Because $r\in F^\times$, for every $\sigma\in\Gal(M/F)$ we have
\[
 \frac{\sigma(w')}{w'}=\frac{\sigma(w)}{w}.
\]
Because the six Kummer classes are independent,
\[
 \Gal(M/F)=
 \langle\sigma_a,\sigma_b,\sigma_c,\sigma_d,\sigma_x,\sigma_y\rangle
 \cong C_p^6,
\]
where each $\sigma_u$ multiplies $\sqrt[p]{u}$ by $\zeta$ and fixes the
other five chosen radicals.

On $F(\sqrt[p]{by})$, the restrictions of $\sigma_b$ and $\sigma_y$ are both
$\rho_{by}$; similarly, on $F(\sqrt[p]{xc})$, the restrictions of
$\sigma_x$ and $\sigma_c$ are both $\rho_{xc}$. Lemma~\ref{lem:W-ratio}
and the four norm equations therefore give
\begin{align*}
 \frac{\sigma_b(w)}{w}
 &=\frac{ax}{f_1^p}\frac{x^{-1}}{f_2^p}
 =\left(\frac{\sqrt[p]{a}}{f_1f_2}\right)^p,\\
 \frac{\sigma_y(w)}{w}
 &=\frac{x^{-1}}{f_2^p}
 =\left(\frac{1}{\sqrt[p]{x}f_2}\right)^p,\\
 \frac{\sigma_c(w)}{w}
 &=\frac{y^{-1}}{f_3^p}\frac{yd}{f_4^p}
 =\left(\frac{\sqrt[p]{d}}{f_3f_4}\right)^p,\\
 \frac{\sigma_x(w)}{w}
 &=\frac{y^{-1}}{f_3^p}
 =\left(\frac{1}{\sqrt[p]{y}f_3}\right)^p,
\end{align*}
and $\sigma_a(w)=\sigma_d(w)=w$.

Hence the Kummer automorphisms extend to $\widetilde L$ by
\begin{align*}
 \widetilde\sigma_a(z)&=z,
 &\widetilde\sigma_d(z)&=z,\\
 \widetilde\sigma_b(z)&=\frac{\sqrt[p]{a}}{f_1f_2}z,
 &\widetilde\sigma_y(z)&=\frac{1}{\sqrt[p]{x}f_2}z,\\
 \widetilde\sigma_c(z)&=\frac{\sqrt[p]{d}}{f_3f_4}z,
 &\widetilde\sigma_x(z)&=\frac{1}{\sqrt[p]{y}f_3}z.
\end{align*}
Let $\tau$ be the automorphism fixing $M$ and satisfying
\[
 \tau(z)=\zeta z.
\]
These lifts generate $\Gal(\widetilde L/F)$ because their restrictions
generate $\Gal(M/F)$ and they contain the generator $\tau$ of
$\Gal(\widetilde L/M)$.

Each displayed lift has order $p$. For example,
\[
 \prod_{i=0}^{p-1}\sigma_b^i
 \left(\frac{\sqrt[p]{a}}{f_1f_2}\right)
 =\frac{a}{N_b(f_1)N_{by}(f_2)}
 =\frac{a}{(ax)(x^{-1})}=1,
\]
and
\[
 \prod_{i=0}^{p-1}\sigma_y^i
 \left(\frac{1}{\sqrt[p]{x}f_2}\right)
 =\frac{1}{xN_{by}(f_2)}=1.
\]
The $c$- and $x$-cases are identical; the $a$- and $d$-lifts plainly have
order $p$.

For commuting quotient automorphisms $\sigma_i,\sigma_j$ with lifts
\[
 \widetilde\sigma_i(z)=q_i z,
 \qquad
 \widetilde\sigma_j(z)=q_j z,
\]
our commutator convention gives
\begin{equation}\label{eq:comm-formula}
 [\widetilde\sigma_i,\widetilde\sigma_j](z)
 =\frac{\sigma_i(q_j)q_i}{\sigma_j(q_i)q_j}z.
\end{equation}
Using \eqref{eq:comm-formula},
\[
 [\widetilde\sigma_a,\widetilde\sigma_b]=\tau,
\]
because $\sigma_a(\sqrt[p]{a})=\zeta\sqrt[p]{a}$ and $\sigma_a$ fixes
$f_1,f_2$. Likewise,
\[
 [\widetilde\sigma_c,\widetilde\sigma_d]=\tau^{-1},
\]
because $\sigma_d(\sqrt[p]{d})=\zeta\sqrt[p]{d}$.
All other commutators among
\[
 \widetilde\sigma_a,\widetilde\sigma_b,
 \widetilde\sigma_c,\widetilde\sigma_d
\]
are trivial.

It remains to check that the two auxiliary lifts generate a commuting
$C_p^2$-factor relative to these four lifts. All commutators are immediate
from the displayed quotients except for
\[
 (\widetilde\sigma_b,\widetilde\sigma_y),\qquad
 (\widetilde\sigma_c,\widetilde\sigma_x),\qquad
 (\widetilde\sigma_x,\widetilde\sigma_y).
\]
For the first pair, $\sigma_b$ and $\sigma_y$ have the same restriction to
$F(\sqrt[p]{by})$, so the two occurrences of the conjugate of $f_2$ in
\eqref{eq:comm-formula} cancel. The second pair is identical. For the last
pair, writing
\[
 q_y=\frac{1}{\sqrt[p]{x}f_2},
 \qquad
 q_x=\frac{1}{\sqrt[p]{y}f_3},
\]
we have
\[
 \sigma_x(q_y)=\zeta^{-1}q_y,
 \qquad
 \sigma_y(q_x)=\zeta^{-1}q_x,
\]
so the factors again cancel.

Thus
\[
 \langle\widetilde\sigma_a,\widetilde\sigma_b,
 \widetilde\sigma_c,\widetilde\sigma_d,\tau\rangle
 \cong\Hp*\Hp,
\]
while $\widetilde\sigma_x$ and $\widetilde\sigma_y$ generate a commuting
$C_p^2$-factor. Since
\[
 [\widetilde L:F]=p\,[M:F]=p^7=|\Hp*\Hp|\,p^2,
\]
these relations account for the full Galois group, and therefore
\[
 \Gal(\widetilde L/F)\cong(\Hp*\Hp)\times C_p^2.
\]
The fixed field of the auxiliary factor contains $K$, because the auxiliary
generators fix $\sqrt[p]{a},\sqrt[p]{b},\sqrt[p]{c},\sqrt[p]{d}$, and its
Galois group over $F$ is $\Hp*\Hp$.
\end{proof}

\begin{remark}
The independence hypothesis in Theorem~\ref{thm:explicit} is used only to
make the two auxiliary Kummer directions honest independent generators of
$\Gal(M/F)$ and thereby produce the transparent direct-product decomposition
above. The existence criterion, Theorem~\ref{thm:global-realizability}, has
no such hypothesis. If $x$ or $y$ is dependent on $a,b,c,d$ modulo
$F^{\times p}$, a radical construction must instead be rewritten relative to
a basis of the actual Kummer space; one cannot simply retain an artificial
$C_p^2$ auxiliary factor.
\end{remark}

\section{Cyclotomic descent}

The preceding explicit formulas were written in Kummer form and therefore
assumed $\mu_p\subset F$. Cyclotomic descent allows the corresponding
existence criterion to be formulated without requiring the ground field to
contain the $p$-th roots of unity. Passing to the cyclotomic extension places
the central embedding problem in Kummer form, where the symbol obstruction
can be computed explicitly, and solvability then descends to the original
field.

Let
\[
 F_p=F(\mu_p),
 \qquad
 d=[F_p:F]\mid p-1,
\]
and put
\[
 G_p=\Gal(F_p/F).
\]
We use the following consequence of the cyclotomic descent theory for
central embedding problems; see \cite[Corollary~8.1.5]{Ledet}.

\begin{theorem}\label{thm:descent}
Let $K/F$ be a finite Galois extension with group $G$ and assume
\[
 K\cap F_p=F.
\]
Let
\[
 1\longrightarrow C_p\longrightarrow E\longrightarrow G\longrightarrow1
\]
be a central non-split extension, and put $K_p=KF_p$. After identifying the
kernel $C_p$ with $\mu_p$ over $F_p$, the embedding problem over $K/F$ is
solvable if and only if the corresponding base-changed embedding problem
over $K_p/F_p$ is solvable.
\end{theorem}

For the $C_p^4$-extensions considered here, the linear-disjointness
hypothesis is automatic. Indeed,
\[
 [K:F]=p^4,
 \qquad
 [F_p:F]\mid p-1,
\]
so $[K:F]$ and $[F_p:F]$ are relatively prime and hence
$K\cap F_p=F$.

\begin{corollary}\label{cor:no-roots-unity}
Let $F$ be a global field, let $p$ be an odd prime with
$p\neq\operatorname{char}F$, and let $K/F$ be a $C_p^4$-extension. Let
$\E=\Hp*\Hp$ and fix an embedding problem
\[
 1\longrightarrow C_p\longrightarrow\E
 \longrightarrow\Gal(K/F)\longrightarrow1
\]
whose quotient generators correspond to the presentation in
Definition~\ref{def:E}. Put $F_p=F(\mu_p)$ and $K_p=KF_p$. Since
$K_p/F_p$ is again a $C_p^4$-extension, choose a Kummer basis adapted to
those four quotient generators,
\[
 K_p=F_p(\sqrt[p]{a},\sqrt[p]{b},\sqrt[p]{c},\sqrt[p]{d}).
\]
Then the embedding problem over $K/F$ is solvable if and only if
\[
 (a,b)_p=(c,d)_p\in\Br(F_p)[p].
\]
In that case there exists a Galois extension $L/F$ containing $K$ with
\[
 \Gal(L/F)\cong\Hp*\Hp.
\]
\end{corollary}

\begin{proof}
The coprimality argument above gives $K\cap F_p=F$, so
Theorem~\ref{thm:descent} applies. Over $F_p$ the extension $K_p/F_p$ is
Kummer, and Corollary~\ref{cor:obstruction-E} identifies the obstruction to
the base-changed problem as
\[
 (a,b)_p-(c,d)_p.
\]
Thus the base-changed problem is solvable exactly when the displayed symbol
classes agree. Cyclotomic descent transfers solvability back to $F$.
\end{proof}

\begin{remark}
When $\mu_p\not\subset F$, one should not write
\[
 K=F(\sqrt[p]{a},\sqrt[p]{b},\sqrt[p]{c},\sqrt[p]{d})
\]
as though $K/F$ were already a Kummer $C_p^4$-extension. The Kummer
description is made only after extending scalars to $F_p=F(\mu_p)$, as in
Corollary~\ref{cor:no-roots-unity}.
\end{remark}






\end{document}